\documentclass{amsart}
\usepackage{amssymb,graphicx,rotating,multirow,multicol}
\usepackage{amsmath,amsthm,amsfonts,amscd}
\usepackage{caption,comment,marginnote}
\usepackage{float}
\usepackage{hyperref}

\usepackage{epigraph}
\swapnumbers
\newtheorem{theorem}{Theorem}[subsection]

\newtheorem{lemma}[theorem]{Lemma}

\newtheorem{proposition}[theorem]{Proposition}
\theoremstyle{remark}
\newtheorem{remark}[theorem]{Remark}
\numberwithin{equation}{subsection}

{\catcode`\@=11 \gdef\mnote#1{\marginpar{\tiny
 \tolerance\@M\spaceskip2.6\p@ plus10\p@ minus.9\p@\rm#1}}}
\def\Dg:{\endgraf{\bf Dg:\enspace}\ignorespaces}

\let\Bbb\mathbb

\def\D{\Delta}
\def\G{\Gamma}

\def\S{\Sigma}
\def\SS{\mathcal S}
\def\SSS{\mathbb S}
\def\T{\mathbb T}

\def\P{\mathbb P}

\def\sm{\smallsetminus}

\newcommand{\be}{\begin{equation}}
\newcommand{\ee}{\end{equation}}
\let\ge\geqslant 
\let\le\leqslant 

\let\til\widetilde

\def\Z{\Bbb Z}
\def\R{\Bbb R}
\def\C{\Bbb C}

\def\Rp#1{\Bbb{RP}^{#1}}

\def\conj{\operatorname{conj}}

\def\dsum{\bot\!\!\!\bot}
\let\+\dsum
\let\-\sm

\let\i=r
\let\j=s

\makeatletter
\newcommand{\addresseshere}{%
  \enddoc@text\let\enddoc@text\relax
}
\makeatother

\title[Affine Cubics]{On Affine Real Cubic Surfaces}

\author[]
{S.~Finashin, V.~Kharlamov}
\address{Middle East Technical University,
Department of Mathematics\endgraf Ankara 06531 Turkey}
\email{serge@metu.edu.tr}
\address{Universit\'{e} de Strasbourg et IRMA (CNRS)\endgraf 7 rue Ren\'{e}-Descartes 67084 Strasbourg Cedex, France}
\email{kharlam@math.unistra.fr}

\makeatletter
\@namedef{subjclassname@2020}{%
  \textup{2020} Mathematics Subject Classification}
\makeatother

\begin{document}
\begin{abstract} We prove that the space of affine, transversal at infinity, non-singular real cubic surfaces has 15 connected components.
We give a
topological criterion to distinguish them and show also how
these 15 components are adjacent to each other via wall-crossing.
\end{abstract}
\maketitle

\vskip-4mm
\setlength\epigraphwidth{.47\textwidth}
\epigraph{On r\'esout les probl\`emes qu'on se pose
et non les probl\`emes qui se posent.}{ Henri Poincar\'e \\}

\section{Introduction}

\subsection{Main task}
We consider an affine 3-space as a chart  $\P^3\- \P^2$ of a projective space $\P^3$ with a fixed hyperplane $\P^2$.
Accordingly, by
an {\it affine cubic surface transversal at infinity}
we mean the complement $X\-A$ where  $X\subset \P^3$ is a projective cubic surface transversal to $\P^2$
and  $A=X\cap\P^2$.
Occasionally  we refer to affine cubics as to pairs $(X,A)$.

The space of non-singular affine cubic surfaces transversal at infinity
is $\P^{19}\-(\D\cup\D')$, where $\P^{19}$ is the space of projective cubic surfaces,
$\D\subset \P^{19}$ is the hypersurface formed by singular surfaces, and $\D'$ is the hypersurface formed by
surfaces which are not transversal to $\P^2$.

Our main objective is to classify up to deformation the
real non-singular affine cubic surfaces transversal at infinity. These surfaces form the real part
$\P^{19}_\R\-(\D_\R\cup\D'_\R)$ of $\P^{19}\-(\D\cup\D')$.
We declare two such surfaces
{\it deformation equivalent} if they belong to the same connected component of $\P^{19}_\R\-(\D_\R\cup\D'_\R)$.

Often, it is convenient to work with a larger space of surfaces, $\P^{19}_\R\-(\overset{\boldsymbol{.}}\Delta_\R\cup\overset{\boldsymbol{.\phantom{a}}}{\Delta'_\R})$,
where the semi-algebraic hypersurface $\overset{\boldsymbol{.}}\Delta_\R\subset \P^{19}_\R$ (resp. $\overset{\boldsymbol{.\phantom{a}}}{\Delta'_\R}\subset \P^{19}_\R$)
is formed by real affine cubic surfaces with a real singular point (resp. not transversal to $\P^2$ at some real point).
Since the both spaces,  $\P^{19}_\R\-(\overset{\boldsymbol{.}}\Delta_\R\cup\overset{\boldsymbol{.\phantom{a}}}{\Delta'_\R})$ and $\P^{19}_\R\-(\D_\R\cup\D'_\R)$, are open in Euclidean topology
and differ by a codimension 2 semi-algebraic set, this does not change the equivalence relation.

\subsection{Deformation classification}

Recall that, for every real algebraic $M$-surface $X$, there exists a
quadratic $\Z/4$-valued function (called {\it Rokhlin-Guillou-Marin quadratic function}) which is
defined on the kernel of the inclusion homomorphism $H_1(X_\R;\Z/2) \to H_1(X_\C; \Z/2)$ and takes value
$2\in\Z/4$
on each real vanishing cycle (see, {\it e.g.} \cite{DK}).

\begin{theorem}\label{AffineSurfDef} There are fifteen deformation classes of real affine
non-singular and transversal at infinity cubic surfaces $X\-A$.
 For all but two exceptional classes, such surfaces
are deformation equivalent if and only if their real parts $X_\R\-A_\R$ are homeomorphic.
The two exceptional classes are those for which:
\begin{itemize}\item
$X$ is an  M-surface, that is $\chi(X_\R)= -5$,
\item
$A$ is an M-curve, that is  $A_\R$ has two components,
\item
both components of $A_\R$ give non-zero classes in $H_1(X_\R)$.
\end{itemize}
The number of real lines intersecting an oval-component $O\subset A_\R$
is 12 for one of these exceptional classes and 16 for another.
The Rokhlin-Guillou-Marin quadratic function $q:H_1(X_\R)\to\Z/4$ takes value
$q(O)=2$ in the first case and $q(O)=0$ in the second.
\end{theorem}

The topological types of $X_\R\-A_\R$ used as a classification invariant in Theorem \ref{AffineSurfDef}
are determined by $X_\R$ (five columns in Table 1), the number of components of $A_\R$ (one or two circles)
and the number of components of $X_\R\-A_\R$ (one, two, or three).
 In Table 1 we list possible combinations and give the corresponding number of deformation classes of $(X,A)$
in each case.

\begin{table}[h!]
\caption{The number of real deformation classes of affine cubics $X\sm A$
depending on the topology of $X_\R$ (columns) and $A_\R$ (rows).
Here $\T^2$ stands for $S^1\times S^1$.}
\scalebox{0.9}{\boxed{
\begin{tabular}{c|c|c|c|c|c}
$X_\R$&$\Rp2\#3\T^2$&$\Rp2\#2\T^2$&$\Rp2\#\T^2$&$\Rp2$&
$\Rp2\+\SSS^2$\\
\hline
$A_\R =  S^1$&1&1&1&1&1\\
$A_\R = S^1\+S^1,\  b_0(X_\R\sm A_\R)=$\scalebox{0.92}{$\begin{cases}1\\2\\3\end{cases}$}&
$\begin{matrix}2\\1\\0\end{matrix}$&$\begin{matrix}1\\1\\0\end{matrix}$&
$\begin{matrix}1\\1\\0\end{matrix}$&
$\begin{matrix}0\\1\\0\end{matrix}$&
$\begin{matrix}0\\1\\1\end{matrix}$
\end{tabular}}}
\end{table}

More concretely:

\begin{itemize}\item
Each of the five classes of $X$ contains precisely
one deformation class of affine cubics $X\-A$ with connected real locus $A_\R$.
\item
The class of $X$ with $X_\R=\Rp2$ contains precisely one deformation class of $X\- A$ with
2-component $A_\R$.
\item
Each of the two classes of $X$ with $X_\R=\Rp2\#k(S^1\times S^1)$, $k=1,2$, contains
precisely
two classes of $X\-A$ with 2-component $A_\R$: for one class the oval of $A_\R$
is null-homologous in $X_\R$ and for another is not.
\item
The class of $X$ with $X_\R=\Rp2\#3(S^1\times S^1)$ contains precisely 3 classes of $X\-A$
with 2-component $A_\R$;  for one class the oval of $A_\R$
is null-homologous in $X_\R$ and for the other two
is not.
For one of the latter two classes the oval is homologous to a real vanishing class in $H_1(X_\R)$
and for another is not.
\item
The class of $X$ with $X_\R=\Rp2\+S^2$ contains precisely two classes of $X\-A$ with
2-component $A_\R$: for one class the components of $A_\R$ lie in different components of $X_\R$,
and for another in the same component $\Rp2$.
\end{itemize}

\subsection{Adjacency of deformation classes}
Two deformation classes are said to be {\it adjacent}
if they meet along
a maximal dimension stratum of $\D_\R \sm \D'_\R$.
These strata are formed by those
transversal at infinity
hypersurfaces which have a node and no other singular points, and they
constitute
the connected components of the smooth part of $\D_\R \sm \D'_\R$.
We call these strata {\it walls} and name the  graph representing the above adjacency relation
the {\it wall-crossing graph}.

We depict the vertices of this graph by circles
labeled inside with the number of real lines intersecting $A_\R$. If $A_\R$ is connected,  it is just the total number of real lines on $X$. If  not, then it is a pair of numbers indicating the number of lines
intersecting the one-sided in $\P^2$ component of $A_\R$ and the two-sided component called {\it the oval}.
The edges are decorated with the number of real lines intersecting
the node that appears at the instance of wall-crossing.
Following the same convention as for vertices, when
$A_\R$ has 2 connected components this number is split into a pair.
In addition, we color in black the vertices representing $(X,A)$ with 2-component $A_\R$ whose oval
bounds a disc in the non-orientable component of $X_\R$.

\begin{theorem}\label{TheoremB}
The wall-crossing graph for affine cubics is as shown on
Fig. \ref{Af2-Graph}: its left-hand-side corresponds to $(X,A)$ with one-component curves $A_\R$,
and the right-hand-side to $(X,A)$ with
two-component $A_\R$.
\end{theorem}

\vskip-3mm
\begin{figure}[h!]
\caption{The wall-crossing graph for transversal at infinity affine cubic surfaces}\label{Af2-Graph}
\hbox{\includegraphics[width=1.2\textwidth]{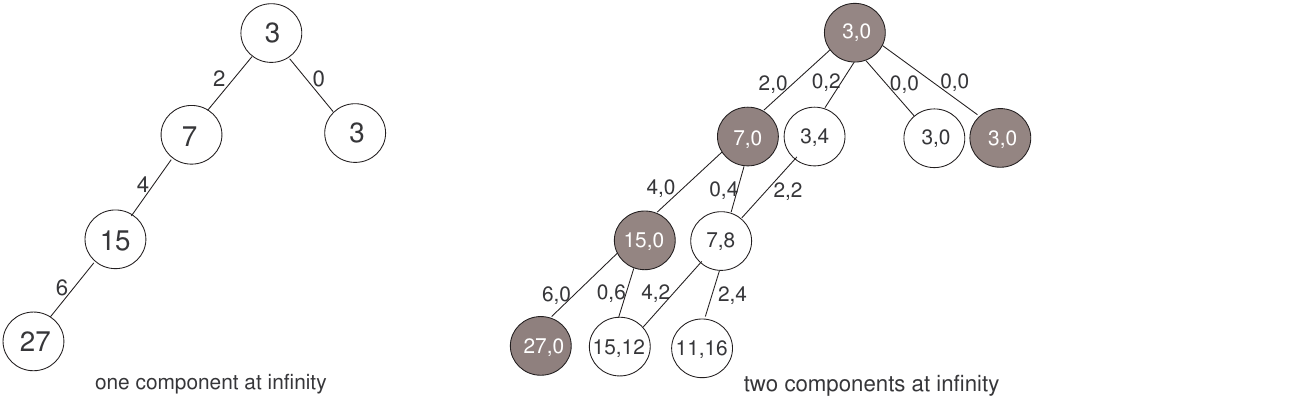}}
\end{figure}
\vskip-1mm

Note that
there may be several walls that separate the same pair of deformation classes and thus, represent the same edge on the graph of Fig. \ref{Af2-Graph} (see
details in Section \ref{ordinary-walls}).
For instance, it is so for
a pair of walls which are adjacent to the same codimension 1 cuspidal stratum of $\D_\R \sm \D'_\R$
(representing $A_2$-singularity on a cubic). This motivated to consider {\it extended walls}
that we define as the connected components
of $(\D_\R \sm \D'_\R)\cup (\D^c_\R\sm \D'_\R)$ where $\D^c$ is the union of the cuspidal strata of $\D$.
We establish that the correspondence between the edges of the graph on Fig. \ref{Af2-Graph}
and extended walls is in fact bijective by combining Theorem \ref{TheoremB} with
the following one.

\begin{theorem}\label{A2-edges-quasisimplicity}
\!Each edge of the graph on Fig.\,\ref{Af2-Graph}
represents just one extended wall.
\end{theorem}

\subsection*{Acknowledgements} This text is extracted from our survey on the topology of real cubic hypersurfaces
that we set up in 2012, but then switched to real enumerative geometry.
It is after a recent communication with V.A.~Vassiliev
that we realized a possible value of publishing separately the results on the classification of real affine cubic surfaces, and we are profoundly grateful to him for this inspiration.

The second author acknowledges support from the grant ANR-18-CE40-0009 of French Agence Nationale de Recherche.


\section{Preliminaries}
\subsection{On the projective real nonsingular cubic surfaces}({\it cf., \cite{Klein},\cite{Segre}})
\begin{theorem}\label{SurfDef} There are five deformation classes of real non-singular projective cubic surfaces.
Two real non-singular projective cubic
surfaces
are deformation equivalent if and only if their real parts $X_\R$ are homeomorphic.
For four of them, $X_\R=\#_{2k+1}\Rp2$, $k=0,1,2,3$.
For the fifth one,
$X_\R=\Rp2 \+S^2$.

\vskip0.05in\hskip-5mm\begin{minipage}{10cm}
\hskip0.2in  The wall-crossing graph of real non-singular projective cubic surfaces
coincides with the left-hand side graph of Figure \ref{Af2-Graph}
and is reproduced to the right.
The number of real lines for each class is encircled and the number of lines
passing through the node corresponding to a wall-crossing decorates the edges.
\end{minipage}\hskip2mm
\begin{minipage}{16cm}
\includegraphics[width=0.15\textwidth]{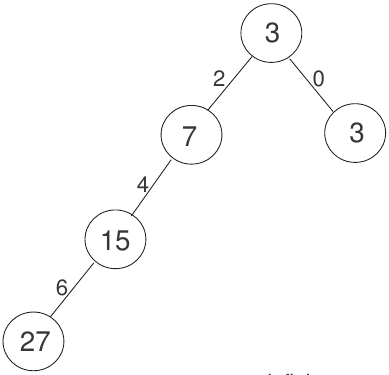}
\end{minipage}
\end{theorem}

Note also that representatives of all 5
deformation classes of real projective nonsingular cubic surfaces can be obtained by small perturbations
of Cayley's 4-nodal cubic surface, $\sum_{i=0}^4 x_i^3 = \frac14 (\sum_{i=0}^4 x_i)^3$ (see \url{https://mathworld.wolfram.com/CayleyCubic.html}).

\subsection{Blowup models of cubic surfaces}({\it cf., \cite{Segre}})
A set of real points or real lines on a real algebraic surface is said to be {\it real} if it is invariant under
the complex conjugation on the surface.

A set of $6$  points $S\subset\P^2$  will be called a {\it typical 6-tuple}
if $S$ does not contain collinear triples and all 6 points are not coconic.

\begin{theorem}\label{cubic-to-plane*}
For any real projective non-singular cubic surface $X$ with
connected real part
there exists a real set of 6 pairwise disjoint lines on $X$.
Blowing down
these lines yields
a real plane with a typical real 6-tuple $S\subset \P^2$.
Conversely, blowing up
$\P^2$ at any real typical 6-tuple $S$
gives a
unique,  up to a real projective transformations,
real projective non-singular cubic surface $X$ with connected real part
equipped
with a real set of 6 skew lines.

This blow-down
transforms isomorphically any real non-singular hyperplane sections $A\subset X$
into a real non-singular plane cubic curves $C\subset\P^2$, $S\subset C$. In this way it establishes a bijection between the set of the former sections
$A$ with the set of the latter
curves $C$.
\qed\end{theorem}

Recall that for a typical 6-tuple $S=\{x_1,\dots,x_6\}\subset\P^2$ the 27 lines in the corresponding cubic surface $X$
include:
\begin{itemize}\item
6 exceptional curves $E_i$ corresponding to blowing up at $x_i$;
\item
$\binom62=15$ proper images
of straight lines $x_ix_j\subset \P^2$, $1\le i<j\le6$;
\item
6 proper images $Q_i$ of plane conics passing through the 5 points $x_j$, $j\ne i$.
\end{itemize}

Theorem \ref{cubic-to-plane*} allows us to represent any real affine nonsingular and transversal at infinity cubic surface $(X,A)$
with connected $X_\R$ via a real non-singular plane cubic curve $C$
equipped with a real typical 6-tuple $S\subset C$.
Since such a representation determines $(X,A)$ only up to a real affine transformation,
which may preserve or reverse orientation,
such approach ignores
a possibility that some affine equivalent cubics may be not deformation equivalent.
By this reason, at our first step of proving Theorem \ref{AffineSurfDef} we classify pairs $(X,A)$ only up to
a weaker equivalence in which
we call two pairs {\it coarse deformation equivalent} if one is deformation equivalent to
the image of another under an affine transformation.

As is well known and easy to show, Theorem \ref{cubic-to-plane*} extends to a wider class of 6-tuples, as well as to families of them.
In particular, the following statement holds.

\begin{theorem}\label{generalized-family-statement} Given a continuous family $(C_t,S_t)$ formed by non-singular real plane cubics $C_t$ and
real 6-tuples $S_t\subset C_t$
such that the linear system of cubic curves passing through $S_t$ has projective dimension 2,
there is a unique, up to a family of real projective transformations, family of real
projective cubic surfaces $X_t$ giving rise
to the given data including the proper image $A_t$ of $C_t$ as a real non-singular hyperplane section of $X_t$.
\qed\end{theorem}

Note that the dimension condition imposed on $S$ in Theorem \ref{generalized-family-statement} is fulfilled if, for example,
$S$ is in a {\it one-nodal} position,
 that is if $S$ lies on a non-singular conic, or if $S$ contains one and only one triple of points aligned.

\subsection{Counting real lines intersecting a given component of $A_\R$}
Here we treat a particular case of blowup models, namely, the case where $C_\R$ has two connected components: an oval and a one-sided component.
We denote by $\SS_{a,b}$, $a,b\ge0$, $a+b\in\{0,2,4,6\}$,
the set of pairs $(C,S)$, where $C\subset\P^2$ is a
non-singular cubic curve, $C_\R$ has two connected components,
$S\subset C$ is a
real typical $6$-tuple
which includes $\mu\le3$
pairs of complex-conjugate imaginary points and $a+b=6-2\mu$ real ones, among which
$a$ points lie on the one-sided component
of $C_\R$ and the remaining $b$ real points lie therefore
on the oval (two-sided component).
By $\S_{a,b}$ we denote the set of real affine cubic surfaces $X\-A$,
represented by pairs $(C,S)\in\SS_{a,b}$, and put
$\S_\mu=\cup_{a+b=6-2\mu}\S_{a,b}$.

\begin{proposition}\label{cubic-to-plane.equiv}
The set $\S_{\mu}$ is partitioned into non-empty subsets as follows:
\begin{itemize}\item
for $\mu=0$ into 3 subsets: $\S_{6,0}$,
$ \S_{0,6}=\S_{3,3}=\S_{4,2}$ and $\S_{1,5}=\S_{2,4}=\S_{5,1}$,
\item for $\mu=1$ into 2 subsets: $\S_{4,0}$ and
$\S_{0,4}=\S_{1,3}=\S_{2,2}=\S_{3,1}$,
\item for $\mu=2$ into 2 subsets $\S_{2,0}$ and
$\S_{0,2}=\S_{1,1}$.
\end{itemize}
The case $\S_3=\S_{0,0}$ is trivial.
\end{proposition}

\begin{proof}
First, note that the indicated subsets are non-empty because
one can
always
displace the required number of real points
on the 2-components of a cubic $C_\R$.
Next, these sets either coincide or disjoint and, thus, give a partition.

If $b\ge3$, we can apply a standard quadratic Cremona transformation which
changes exceptional divisors over points $x_1,x_2,x_3$ on the oval $O\subset C_\R$
by the proper transform of lines $x_1x_2$, $x_2x_3$ and $x_3x_1$.
This Cremona transformation takes
the oval with $b$ points into a one-sided
component with $b-3$ points and, thus, identifies $\S_{a,b}$ with $\S_{b-3,a+3}$. Hence,
$\S_{3,3}=\S_{0,6}$, $\S_{2,4}=\S_{1,5}$ for $\mu=0$ and
$\S_{1,3}=\S_{0,4}$ for $\mu=1$.

If $a\ge2$ and $b\ge1$,
we apply the same Cremona transformation but with
$x_1$ chosen on the oval, and $x_2$ with $x_3$ on the one-sided component.
This
also takes the oval into one-sided component and, thus, identifies $\S_{a,b}$ with $\S_{b+1,a-1}$.
Hence,
$\S_{5,1}=\S_{2,4}$, $\S_{4,2}=\S_{3,3}$ for $\mu=0$, and
$\S_{3,1}=\S_{2,2}$ for $\mu=1$.

If $b\ge1$ and $\mu\ge1$, we apply again the same transformation but with $x_2,x_3$ chosen imaginary complex conjugate and $x_1$ on the oval.
This also takes the oval into one-sided component and, thus, identifies $\S_{a,b}$ with $\S_{b-1,a+1}$.
Therefore,
$\S_{3,1}=\S_{0,4}$, $\S_{2,2}=\S_{1,3}$ for $\mu=1$, and
$\S_{1,1}=\S_{0,2}$ for $\mu=2$.
\end{proof}

To distinguish affine cubics $(X,A)$ by counting real lines intersecting  a given component of
$A_\R$, we make first a count in terms of $(\P^2,C,S)$ and then translate it in terms of $(X,A)$.

\begin{lemma}\label{C-enumeration}
If $(C,S)\in\SS_{a,b}$, then
the proper image $\til O\subset A_\R$ of the oval $O\subset C_\R$
is intersected by $2b+ab$ real lines of $X$
if $b$ is even, and by $6-2\mu+ab$ if $b$ is odd.
\end{lemma}

\begin{proof}
For each of the $b$ points $x_i\in O$, the component $\til O$ is intersected by
the line $E_i$ and by the proper images of lines
$x_ix_j$ for each of the $a$ points $x_j\in S$ on the one-sided component $A_\R\-O$.
In addition, $\til O$ intersects $b$ lines $Q_i$ if
$b$ is even, or $a=6-2\mu-b$ lines $Q_j$ if $b$ is odd.
\end{proof}

\begin{lemma}\label{oval-comparison}
Assume that $(C,S)\in\SS_{a,b}$. Then,
the proper image $\til O\subset A_\R$ of the oval $O\subset C_\R$ is an oval of $A_\R$
if and only if $b$ is even.
\end{lemma}

\begin{proof}
This is because
after blowing up at a point of a curve on a surface, the tubular neighbourhood of
 the proper
 image of the curve alternates its orientability.
\end{proof}

\begin{proposition}\label{line-count} Assume $(X,A)\in\S_{a,b}$. Then
the oval of $A$ intersects
$0$ lines if $b=0$. If $b>0$, then this oval intersects:
\begin{itemize}
\item
12 lines for $\S_{0,6}=\S_{3,3}=\S_{4,2}$ and 16 lines for
$\S_{1,5}=\S_{2,4}=\S_{5,1}$;
\item
8 lines if  $a+b=4$;
\item
4 lines if $a+b=2$.
\end{itemize}
\end{proposition}

\begin{proof}
Lemma \ref{C-enumeration} gives answer $2b+ab$
for even $b$, which is the required number by Lemma \ref{oval-comparison}.
For odd $b$, these Lemmas imply that
the one-sided component $\til O$ intersects $6-2\mu+ab$ lines and, thus,
the oval of $A_\R$ intersects the remaining real lines.
\end{proof}

\section{Proof of the Main Theorem}

\subsection{Coarse deformation classes via blow-up models}

\begin{lemma}\label{through-nodal} If
the triples $(\P^2_\R, C^1_\R, S^1_\R)$, $(\P^2_\R,C^2_\R, S^2_\R)$ are homeomorphic, then
the associated with them real affine cubic surfaces
$(X^1,A^1)$,
$(X^2,A^2)$ are
coarse deformation equivalent.
\end{lemma}

\begin{proof} According to Theorem \ref{generalized-family-statement} to prove coarse deformation equivalence of  $(X^1,A^1)$ and $(X^1,A^2)$, it is sufficient to build
a real deformation between $(C^1_\R, S^1_\R)$ and $(C^2_\R, S^2_\R)$ using at worth one-nodal 6-tuples.
Such a real deformation equivalence can be built in three steps:
deforming $C^1$ to $C^2$, transporting $S^1$ from $C^1$ to $C^2$ along a chosen deformation by means of a family of typical 6-tuples, and finally moving the transported $S^1$ to $S^2$. A real deformation between $C^1$ and $C^2$
exists due to the deformation classification of real plane cubic curves. Given two typical $\conj$-equivariant 6-tuples $S^1, S^2$ on the same real
non-singular plane cubic curve $C$ we may join them by a generic pass $S^t$, $t\in [1,2]$, if
the pairs $(C_\R, S^1_\R)$, $(C_\R, S^2_\R)$
the pairs are homeomorphic. It may happen that at a finite number
of times the 6-tuples $S^t$ go through a one-nodal position. But since in the associated family of real surfaces $(X,A)$ the topology of the real part of the cubic surfaces is not
changing, and since when $X$ acquires a node the curve $A$ does not pass through the node
(and remains non-singular), we may conclude that at each nodal instance
we touch the discriminant $\Delta^a$ without crossing it, and thus, due to the deformation classification of real non-singular projective cubic surfaces, may bypass the family of surfaces by a small real deformation.
\end{proof}

\begin{proposition}\label{A-2comp}
The real transversal at infinity affine cubics $(X,A)$ for which both $X_\R$ and $A_\R$ are connected form
$4$ coarse deformation classes: each class is determined by
the deformation class of $X$.
\end{proposition}

\begin{proof}
Using Theorem \ref{cubic-to-plane*} we interpret the claim in the language of blowup models of $(X,A)$,
 that is triples $(\P^2,C,S)$. Since, as it follows from Theorem   \ref{SurfDef}, the class of $X$ is determined by $\mu\in\{0,1,2,3\}$,
 for construction of $A$ for a fixed class of $X$ it is enough to pick a connected cubic $A_\R$
and place on it $6-2\mu$ real points (and $\mu$ pairs of imaginary points on $A\-A_\R$).
 For uniqueness of the coarse deformation class of such pairs $(X,A)$ it is enough to
apply Lemma \ref{through-nodal}.
\end{proof}

\begin{proposition}\label{8-classes}
The real transversal at infinity affine cubics $(X,A)$ for which $X_\R$ is connected, while $A_\R$ is not, form
8 coarse deformation classes.  For all but 2 exceptional cases, indicated in Theorem \ref{AffineSurfDef},
the coarse deformation class is determined by
the topology of $X_\R\-A_\R$, while the exceptional cases are distinguished as stated in Theorem \ref{AffineSurfDef}.
\end{proposition}

\begin{proof}
We use Theorem \ref{cubic-to-plane*} and  Lemma \ref{through-nodal} like in the case of connected $A_\R$,
and then apply Proposition \ref{cubic-to-plane.equiv} giving 8 coarse deformation classes of $(X,A)$:
three for $\mu=0$, two for each of $\mu\in\{1,2\}$ and
one for $\mu=3$. Proposition \ref{line-count} gives the numbers of lines intersecting each component of $A_\R$ and in particular, it describes the two exceptional classes for $\mu=0$
in terms of real lines, as it is stated in Theorem  \ref{AffineSurfDef}.
To obtain the other description, in terms of Rokhlin-Guillou-Marin
quadratic function, it is sufficient to observe on examples (see Figure \ref{sect1})
that in the case
$\S_{0,6}=\S_{3,3}=\S_{4,2}$,
 (where the number of real lines intersecting the oval is 12),
the oval is a vanishing cycle, and in the other case,
$\S_{5,1}=\S_{2,4}=\S_{1,5}$,
(where the number of real lines intersecting the oval is 16),
the oval is homologous to the sum of two disjoint vanishing cycles.
Therefore, in the first case the quadratic function takes value $2\in\Z/4$ on the oval, while in the other
case its value is $0\in\Z/4$.
\end{proof}

\vskip-4mm
\begin{figure}[h!]
\caption{}\label{sect1}
\begin{minipage}{6.8cm}\vskip-5mm
\hbox{\includegraphics[width=0.3\textwidth]{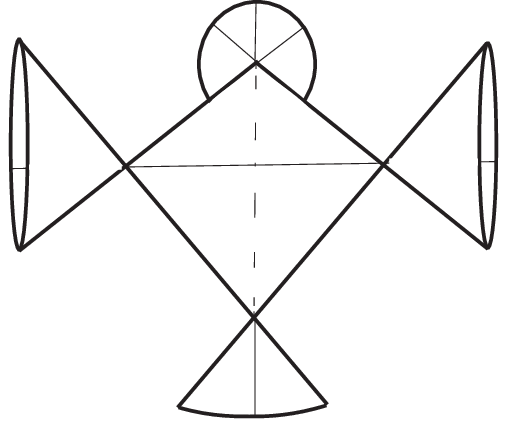}}
\hbox{\includegraphics[width=0.3\textwidth]{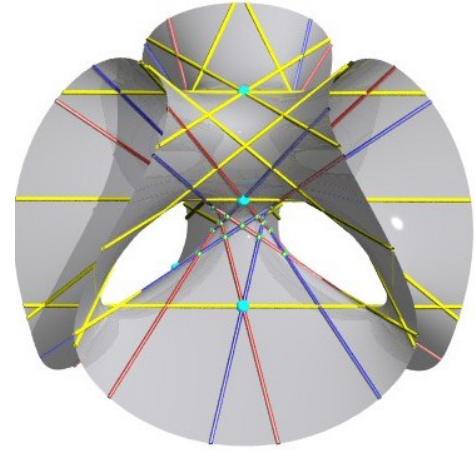}}
\end{minipage}
\begin{minipage}{10cm}
{\small
Cayley's 4-nodal cubic surface with a tetrahedron-like piece $T$ (top figure)
gives after a small perturbation a cubic surface $X$ (at the bottom), with nodes replaced by 1-handles.
A plane truncating  $T$ near its vertex traces on $X_\R$ a cubic curve $A_\R$ whose oval
is a vanishing cycle.
Another plane separating one pair of vertices of $T$ from another traces
on $X_\R$ a cubic curve $A_\R$
whose oval is homologous to the sum of two disjoint vanishing cycles.
In each case the oval intersects $4n$ lines in $X_\R$,
where $n$ is the number of edges of $T$  intersected by the plane (resp. 3 and 4).}
\end{minipage}
\end{figure}

\subsection{Affine cubics with disconnected $X_\R$}
\begin{lemma}\label{Petrovsky-hyperbolic}
If a real cubic hypersurface $X_\R\subset \P^{n+1}_\R$ has at least one contractible $n$-dimensional component in $\P^{n+1}_\R$, then $X_\R$ has no singular points.
\end{lemma}
\begin{proof} Let $F\subset X_\R$ be a contractible $n$-dimensional component. Then it decomposes $\P^{n+1}_\R$ into
$\ge2$ connected components so that
at least
one of them is  contractible in $\P^{n+1}_\R$. Pick a point inside this contractible component.
Each real line passing through this point intersects
$F$ at an even number $\ge2$ of real points, and in addition the other component of $X_\R$
at $\ge1$ points.
So, existence of a singular point contradicts to the B\'ezout theorem.
\end{proof}

\begin{proposition}\label{3-classes}
The real transversal at infinity affine cubic surfaces $(X,A)$ with disconnected $X_\R$ form
3 deformation classes.  One is formed by cubics with connected $A_\R$ and two with disconnected.
For cubics of one of the latter two classes, the oval of $A_\R$ is contained in the spherical component of $X_\R$,
while for the other class it is contained
 in the non-spherical component.
\end{proposition}

\begin{proof}
We start with the case in which
the spherical component $F\subset X_\R$ contains an oval $\G=F\cap \P^2_\R$
of $A_\R=X_\R\cap \P^2_\R$.
Then, we select as a standard model a
pair $(X', A')$ obtained by a small perturbation of the product of a sphere with a plane so that the spherical component $F'$ of $X'_\R$ and the
oval component of the curve $A'_\R$ are contained inside  $F$ and, respectively, $\Gamma$, while the non-spherical component of $X'_\R$ lies outside $F$.
Next, we pick cubic equations $f$ for $X_\R$ and $g$ for $X'_\R$
so that  $fg>0$ inside $F'$.
Then, by Bolzano intermediate value theorem
the cubic surface $X_\R(t)$ defined by $tf+(1-t)g$ has
for any $t\in [0,1]$ at least one 2-dimensional
component contained inside $F$.
Similarly, $A_\R(t)=X_\R(t)\cap \P^2_\R$ has
at least one 1-dimensional component inside $\Gamma$.
These components have to be contractible and so, by Lemma \ref{Petrovsky-hyperbolic},
the pairs $(X_\R(t), A_\R(t))$ are nonsingular for all $t\in [0,1]$. Thus,
$(X_\R, A_\R)$ and $(X'_\R, A'_\R)$ are deformation equivalent, and there remains to notice that all the standard models
$(X',A')$ are deformation equivalent to each other.

In the other 2 cases
we pick as a standard model a pair $(X',A')$ so that $X'$ is defined by
$g=wq+\epsilon f_\infty$, $0<|\epsilon |<\!\!<1$, where
$\P^2=\{w=0\}$, $f_\infty=f(x,y,z,0)$ ($f$ still defines $X$),
and $q=0$  defines a small sphere contained inside $F$ and chosen so
that $fg>0$ inside $F'$.
By the same arguments as above,
 the cubic surfaces $X(t)$  defined by $tf+(1-t)g=0$ are
 all non-singular.  In addition,
 $A(t)=A=A'$ for each $t$.
So, it remains to notice that all the standard models
$(X'_\R, A'_\R)$ with
homeomorphic $A'_\R$ are deformation equivalent, since all real nonsingular cubic curves with the same topology are deformation equivalent.
\end{proof}


\subsection{Achirality and end of proof of Theorem \ref{AffineSurfDef}}\label{achiral}
To finish the proof of Theorem \ref{AffineSurfDef} there remains to check achirality in each of coarse deformation classes.
\begin{proposition}\label{strong-achirality}
Each coarse deformation class of real affine surface $(X,A)$
contains a surface invariant under some real affine reflection.
\end{proposition}

\begin{proof} The Cayley 4-nodal real cubic surface $\sum_{i=0}^4 x_i^3 = \frac14 (\sum_{i=0}^3 x_i)^3$
is invariant under the projective transformations induced by
permutations of the coordinates $x_0,\dots, x_3$.
In particular, it is invariant under 6 reflections induced by
the $6=\binom{4}2$ transpositions.
So, we may obtain examples of achiral real affine surfaces just by choosing one of these reflections and considering a reflection invariant real perturbation
of the Cayley surface and a reflection invariant real projective plane as the plane at infinity. It is then straightforward to construct in such a way a representative for each of our 15 coarse deformation
classes. (In particular, the both plane sections indicated in caption to Figure \ref{sect1} can be chosen invariant with respect to a reflection fixing 2 of 4 vertices of the tetrahedron and permuting the 2 others.)

Recall that for all but 2 exceptional cases a coarse deformation class is determined by simple topological properties of the pair $(X_\R, A_\R)$. As to the 2 exceptional cases already the construction
indicated in the proof of Proposition \ref{8-classes} provides reflection invariant examples.
\end{proof}

\subsection{Proof of Theorem \ref{TheoremB}}
The forgetful map $(X,A)\mapsto X$ induces a projection $\Gamma^{\rm af}\to \Gamma$ of the wall-crossing graph $\Gamma^{\rm af}$ of real transversal at infinity affine cubic surfaces to the wall-crossing graph $\Gamma$ of projective real nonsingular cubic surfaces that is shown in Theorem \ref{SurfDef}. For the connected component of $\Gamma^{\rm af}$ that describes adjacency for deformation classes of affine cubics $(X,A)$ with connected $A_\R$, this projection is an isomorphism, as it follows from its bijectivity at the level of vertices established in Proposition \ref{A-2comp}.

For the remaining part of $\Gamma^{\rm af}$ (that describes adjacency in the cases with 2-component $A_\R$), the set of vertices is described in Propositions \ref{8-classes}
and \ref{3-classes}. Labels of the vertices are determined by Proposition \ref{line-count}.
To determine the edges and their labels,
we use
the following observation: if under wall-crossing the Euler characteristic of $X_\R$ is increasing, then the number of real lines is decreasing by $0\le 2k\le 12$
(where $k$ is the number of real lines passing through the node at the instance of wall-crossing) and, in particular,  each of the labels $a$ and $b$ decreases.
Existence of the edges with labels $(a,b)$ such that $a,b\in \{0,2,4,6\}, a+b\le 6$ follows from Theorem \ref{generalized-family-statement} and
known examples of transversal pairs of conic and cubic,
see \cite[Section 9.2]{FK}.

\section{Edges of the wall-crossing graph}\label{preliminaries}
\subsection{Extended walls}
Let us assume that an affine cubic
 $X\sm A, A=X\cap \P^3$, has a double point $x\in X\sm A$ and consider affine coordinates centered at $x$. Then, such an affine cubic is defined by equation $f_2+f_3=0$ where
 $f_2$ and $f_3$ are nonzero
 ternary quadratic and cubic forms: $f_3$ defines the (cubic) curve $A\subset \P^2$ and $f_2$ a conic.
The following criterion
is straightforward from definitions ({\it c.f.}, for example, \cite[Lemma 2.2]{4fold}).

\begin{lemma}\label{transversality}
An affine cubic surface $\{f_2+f_3=0\}$ is transversal at infinity and has no singular point except
the node or cusp at the coordinate origin,
if and only if $\{f_3=0\}\subset \P^2$ is a nonsingular cubic transversal to the conic $\{f_2=0\}\subset\P^2$.
\qed\end{lemma}

\begin{proposition}\label{edges-to-pairs}
Pick any real nonsingular plane cubic curve $C$ and any label: ``$a$'' if $C_\R$ is connected, or ``$a,b$'' if it has two components.
Then the above construction gives a bijection between the set of extended walls representing
the edges of graphs on Figure \ref{Af2-Graph}
with the chosen label and the connected components of the space of real conics $B$ intersecting
transversely cubic $C$, with the number of intersection points
on each component of $C_\R$ prescribed by the label.
\end{proposition}

\begin{proof} It is a straightforward consequence of Lemma \ref{transversality}, Theorem \ref{TheoremB}, and connectedness of the group of real affine transformations preserving orientation.
\end{proof}

\begin{remark} The only label marking more than one edge on Figure \ref{Af2-Graph}
is \scalebox{0.9}{``$0,0$''}.
\end{remark}

\subsection{Auxiliary covering}\label{aux}
Let $C$ be a nonsingular plane cubic curve. Denote by $C^{(n)}$ its $n$-fold symmetric power and by
${\dot{C}}^{(n)}$
its Zariski
open subset formed by collections of $n$ distinct points.
For any effective divisor of degree 5 on $C$, which we consider as a point $D\in C^{(5)}$, there is a unique conic $B$ that cuts on $C$ a divisor $B\cdot C \ge D$. This correspondence, $D\mapsto B\cdot C\in C^{(6)}$,
defines a regular map $f:C^{(5)}\to C^{(6)}$ whose image $V=f(C^{(5)})\subset C^{(6)}$ is a codimension 1 smooth subvariety of  $C^{(6)}$  (in fact, this image is isomorphic to $\P^5$ due to its identification
with the space of plane conics).

In accordance with Proposition \ref{edges-to-pairs}, we are especially interested in a Zariski open subset of $V$ defined as
$\dot{V} =f(C^{(5)})\cap\dot{C}^{(6)}=f(\dot{ C}^{(5)})\cap\dot{C}^{(6)}$.

If the cubic curve $C$ is real,
the (open) manifold $\dot{C}^{(n)}_\R=\{D\in \dot{C}^{(n)}: \conj D=D \}$ splits
into several connected components enumerated by the number of points in  $D\in \dot{C}^{(n)}_\R$
on each of the connected components of $C_\R$.
Namely,
we have a partition
$$ \dot{C}^{(n)}_\R= \begin{cases}\ \ \displaystyle
\bigsqcup_{\i\ge 0,\ \i=n\, mod\, 2}\ \dot{C}^{\i}_\R,
 \qquad\text{if $C_\R$ is connected,} \\
\displaystyle\bigsqcup_{\i,\j\ge 0,\ \i+\j=n\, mod\, 2}\hskip-2mm \dot{C}^{\i,\j}_\R\qquad\text{if $C_\R$ has two components,}
\end{cases}$$
where $\i$ stands for the number of points of $D\in \dot{C}^{(n)}_\R$ on the 1-sided component of $C_\R$ and $\j$ for that on the oval.

In the case $n=6$, for each $a=0\mod2, a\ge 0$ and, respectively, each $a+b=0\mod 2$ with $a,b\ge 0$, we put
$$
\dot{V}^{a}_\R= V\cap  \dot{C}^a_\R, \quad \dot{V}^{a,b}_\R= V\cap  \dot{C}^{a,b}_\R.
$$

The above map $f$ respects
this partition. It restricts to a collection of maps
$$\begin{cases}
\dot{f}
: \ \dot{C}^{\i}_\R\to V_\R,\ \quad\text{if $C_\R$ is connected},\\
\dot{f} :
\dot{C}^{\i,\j}_\R\to V_\R,\quad\text{if $C_\R$ has two components.}
\end{cases}
$$
Note that $\dot{f}(\dot{C}^\i_\R)$ and $\dot{f}(\dot{C}^{\i,\j}_\R)$
are contained in the closure of $\dot{V}^{\i+1}_\R$
 and, respectively,
$\dot{V}^{\bar \i, \bar \j}_\R$, where
for $x=\i,\j$ we let
$\bar x:= x+1$ if $x$ is odd and $\bar x=x$ if even.

\begin{lemma}\label{R-branching}
Each restriction map
$\dot{f} :\dot{C}^{\i}_\R\to V_\R$  or
$\dot{f} : \dot{C}^{\i,\j}_\R\to V_\R$ is smooth
with only fold singular points.
Namely,
$\dot{f}$ is a local diffeomorphism at
a point $D$ if $\dot{f}(D)\in \dot{V}_\R$, while for
$\dot{f}(D)\in V_\R\sm\dot{V}_\R$ the map $\dot{f}$ is given
by $\dot{f}(x_1,\dots, x_4, y)=(x_1,\dots, x_4, y^2)$ in appropriate local coordinates of the domain and codomain.
Moreover, $\dot{f}^{-1}(\dot{f}(D))=D$ for every fold point $D$ (a point with $\dot{f}(D)\in V_\R\sm\dot{V}_\R$).
\qed
\end{lemma}

\begin{proof}
Straightforward from the Abel-Jacobi theorem implying that
for each 5-tuple $D=\{p_1,\dots,p_5\}$ in the domain of $\dot{f}$ the 6th point in the 6-tuple $\dot{f}(D)=\{p_1,\dots,p_6\}$
is determined by the condition $p_1+\dots+p_6=0$ with respect to
the group structure in $C$ with a flex-point of $C$ taken for zero.
\end{proof}

The following result is an immediate consequence of Lemma \ref{R-branching}.

\begin{proposition}\label{interior}
Each image $\dot{f}(\dot{C}^{\i}_\R)$,
$\dot{f}(\dot{C}^{\i,\j}_\R)$ is a manifold with the boundary formed by
fold-points of $\dot{f}$
and the interior part being $\dot{V}^{\i+1}_\R$, $\dot{V}^{\bar \i, \bar \j}_\R$ respectively.
\qed\end{proposition}

\subsection{Proof of Theorem \ref{A2-edges-quasisimplicity}}\label{connectedness}
Let us treat, first, the case of edges with labels
$a$ and $a,b$
different from $0$ and $0,0$.
In such a case, according to Proposition  \ref{edges-to-pairs}, it is sufficient to show that each of the manifolds
$\dot{V}^a_\R$ and $\dot{V}^{a,b}_\R$ is connected.

To prove their connectedness, we put $a=\i +1$ or $a,b=\bar\i,\bar\j $,
respectively, and
deduce connectedness of $\dot{f}(\dot{C}^{\i}_\R)$ and $ \dot{f}(\dot{C}^{\i,\j}_\R)$ from
connectedness of $\dot{C}^{\i}_\R$ and $\dot{C}^{\i,\j}_\R$ (already observed in Section \ref{aux}).
By virtue of Proposition  \ref{interior}, this implies connectedness of $\dot{V}^{\i+1}_\R=\dot{V}^a_\R$ and $\dot{V}^{\bar \i, \bar \j}_\R=\dot{V}^{a,b}_\R$.

For the case of labels $0$ or $0,0$,
we observe that
the set of real quadratic forms
taking a fixed sign on each component
of $C_\R$ is convex and that
non-transversality with $C$ is a codimension 2 condition (since it happens simultaneously at pairs of complex-conjugate points).
Due to Proposition  \ref{edges-to-pairs}, this gives
one extended wall for connected $C_\R$ and two walls for disconnected. This corresponds to
one edge
labeled with $0$ and, respectively,
two edges labeled $0,0$ on Figure \ref{Af2-Graph}.

\section{Concluding remarks}
\subsection{Intersecting a cubic by curves of arbitrary degree}
Our proof of Theorem \ref{A2-edges-quasisimplicity} given in Section \ref{preliminaries} is based on
enumeration of connected components of the space of real conics intersecting transversally a fixed nonsingular real cubic $C\subset\P^2$.
If we replace conics by curves of arbitrary degree $d\ge1$,
we may similarly define a map $f:C^{(3d-1)}\to C^{(3d)}$, let $V=f(C^{(3d-1)})$,
consider a similar partition
of $\dot{V}_\R$ into $\dot{V}^{a}$ or $\dot{V}^{a,b}$, and prove connectedness of
$\dot{V}^{a,b}$  for $(a,b)\ne(0,0)$ and
of $\dot{V}^{a}$ for any $a$.

The case $(a,b)=(0,0)$ is
possible only
if $d$ is even. Then the same sign and convexity arguments as for $d=2$ apply and show that for any even $d$ the space $\dot{V}^{0,0}$, that is the space of real nonsingular plane curves of degree $d$ transversal to $C$ and having no
real points common with $C$, consists of 2 connected components. The latter ones are distinguished by comparison
of the signs that an underlying real form of degree $d$ takes on the real components of $C$.

\subsection{Number of ordinary walls
in an extended wall}\label{ordinary-walls}
Recall, that by definition (see Introduction) the "ordinary" walls are the connected components of $\D_\R \sm \D'_\R$, while
the extended walls are the connected components of a bigger space obtained by adding the cuspidal strata.
As it is natural to expect,
 there exist extended walls that contain more than one ordinary wall.
A complete answer is given in Table \ref{table} that
indicates the precise
number of ordinary walls in each of the extended walls. The latter ones are marked with the labels as in the previous section (like the corresponding
edges on Figure \ref{Af2-Graph}). Note that label $0,0$ is attributed to two different extended walls, so, ``1'' under ``$0,0$''
says that each of them contains just one ordinary wall.

\begin{table}[h!]
\caption{Number of ordinary walls in each extended wall}\label{table}
\vskip-3mm
\scalebox{0.94}{\boxed{
\begin{tabular}{c|c|c|c|c|c|c|c|c|c|c|c|c|c|c}
\!\!extended wall label\!&0&2&4&6&0,0&0,2&2,0&0,4&2,2&4,0&0,6&2,4&4,2&6,0\\
\hline
number of walls&1&1&1&3&1&1&2&1&1&2&1&2&1&3
\end{tabular}}}
\end{table}

In fact, the set of ordinary walls is
in a natural bijective correspondence with
the set of deformation classes
of real plane quintics that split into a nonsingular cubic $C$ and nonsingular conic $B$
meeting each other transversally ({\it cf.} Lemma \ref{transversality}).
Therefore,
the results shown in the table can be obtained from \cite[Section 9.2]{FK}, where the latter classification was presented as follows.

\vskip-3mm
\begin{figure}[h!]
\caption{The extremal mutual positions of a cubic and a conic} \label{2+3}
\hbox{\includegraphics[width=1\textwidth]{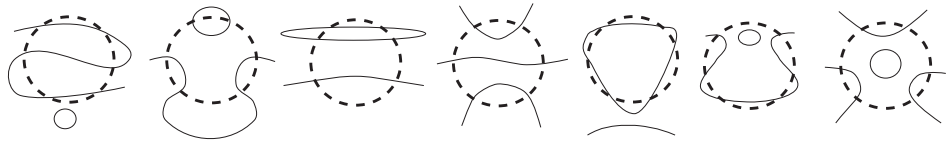}}
\end{figure}
\vskip-3mm

On one hand each deformation class is determined by the topology of the
mutual arrangement of $C_\R$ and $B_\R$ in the real projective plane. On the other hand,
as it was shown by G.~Polotovsky \cite{Pol},
 there are 25 such arrangements:
7 extremal classes  are shown in Figure \ref{2+3}, while the other, non-extremal, cases
are obtained from the extremal ones by repeating
the following moves. One move consists in
erasing an oval (of the cubic or of the conic)
containing no intersection points,
and
the other in shifting a piece of curve containing
a pair of consecutive (both on the conic and the cubic) intersection points,
so that these two intersections disappear.

\subsection{Vasilliev's conjectures}
The deformation classification in this paper
is naturally related to enumeration of connected components of
non-discriminant
perturbations of
parabolic
singularities
of type $P_8$. In a recent preprint \cite{Vas}  V.A.~Vasilliev  provided lower bounds on the number of such components for each of the 8 classes of parabolic singularities and conjectured that his bounds are sharp. Our results confirm partially this conjecture.

\end{document}